\documentclass[a4paper,12pt,leqno]{amsart}
\usepackage{a4wide}
\numberwithin{equation}{section}
\usepackage{enumitem}
\usepackage[utf8]{inputenc}
\usepackage{hyperref}
\usepackage{amssymb, amsmath, amsthm}
\usepackage{amsfonts} 
\usepackage{color}
\usepackage{mathrsfs}
\usepackage{graphicx}
\usepackage{comment}
\usepackage{dcolumn}
\usepackage{bm}
\usepackage[hang,small,bf]{caption}
\usepackage[subrefformat=parens]{subcaption}
\usepackage{algorithmic}
\usepackage{algorithm}
\numberwithin{equation}{section}
\newcommand{\e}{\varepsilon}
\renewcommand{\epsilon}{\varepsilon}
\renewcommand{\d}{\mathrm{d}}
\newcommand{\R}{\mathbb R}
\newcommand{\N}{\mathbb N}

\newcommand{\acosh}{\mathrm{arcosh}}
\newcommand{\asinh}{\mathrm{arsinh}}

\renewcommand{\AA}{A} 

\newcommand{\vmean}{\mathbf{M}}  
\newcommand{\smean}{\mathbf{m}}  

\newtheorem{thm}{Theorem}[section]
\newtheorem{lem}[thm]{Lemma}

\theoremstyle{definition}

\newtheorem{rmk}[thm]{Remark}
\newtheorem{ex}[thm]{Example}

\providecommand{\eat}[1]{}

\newcommand{\Hm}[1]{\leavevmode{\marginpar{\tiny%
$\hbox to 0mm{\hspace*{-0.5mm}$\leftarrow$\hss}%
\vcenter{\vrule depth 0.1mm height 0.1mm width \the\marginparwidth}%
\hbox to 0mm{\hss$\rightarrow$\hspace*{-0.5mm}}$\\\relax\raggedright
#1}}}

\title[]{Construction of signed distance functions through\\ an elliptic equation}

\author[T.~Hasebe]{Takahiro Hasebe}
\address[Takahiro Hasebe]{Department of Mathematics, \\
Hokkaido University, \\
Sapporo 060-0810, Japan}
\email{thasebe@math.sci.hokudai.ac.jp}

\author[J.~Masamune]{Jun Masamune}
\address[Jun Masamune]{Graduate School of Informatics, \\Kyoto University, \\
Kyoto 606-8501, Japan}
\email{masamune.jun.3i@kyoto-u.ac.jp}

\author[T.~Oka]{Tomoyuki Oka}
\address[Tomoyuki Oka]{
Department of Intelligent Mechanical Engineering, \\
Fukuoka Institute of Technology,\\
Fukuoka, 811-0295, Japan}
\email{t-oka@fit.ac.jp}

\author[K.~Sakai]{Kota Sakai}
\address[Kota Sakai]{Graduate School of Engineering, \\
The University of Tokyo,\\
Bunkyo-ku 113-8656, Japan}
\email{sakai-kota786@g.ecc.u-tokyo.ac.jp}

\author[T.~Yamada]{Takayuki Yamada}
\address[Takayuki Yamada]{Graduate School of Engineering, \\
The University of Tokyo,\\
Bunkyo-ku 113-8656, Japan}
\email{t.yamada@mech.t.u-tokyo.ac.jp}

\date{\today}

\subjclass[2020]{Primary: 35J25 Secondary: 35B40, 49Q10}
\keywords{signed distance function, elliptic equation, two-phase domain, rate of convergence, quantitative analysis}

\thanks{
J.M.~is partially supported by LUPICIA CO., LTD and JSPS KAKENHI Grant Number JP23H03798.
T.O.~is partially supported by JSPS KAKENHI Grant Number JP22K20331 and 23K12997 and a project JPNP20004 subsidized by the New Energy and Industrial Technology Development Organization (NEDO). 
K.S. is partially supported by JST SPRING Grant Number JPMJSP2108.
T.H., J.M., T.O. and T.Y. are partially supported by JSPS KAKENHI Grant Number JP23H03800. The authors are grateful to Hirotoshi Kuroda for discussions in the early stage of this work and to the anonymous referees for constructive helpful suggestions.}

\begin{document}

\begin{abstract}
Motivated by recent advances in structural optimization, we propose a novel method for constructing the distance function to the boundary of a given domain. Building on and extending the celebrated Varadhan asymptotic theory, our approach reformulates the governing equation into a more appropriate framework. A central contribution of this work is the derivation of convergence rates within this new setting, which are shown to be optimal in one dimension and offer significant improvements over existing results in higher dimensions. 
\end{abstract}

\maketitle

\section{Introduction}

The main aim of this paper is to propose a new method for constructing the \emph{distance function} and to rigorously establish its convergence behavior---most notably, a \emph{sharp convergence rate}, which is \emph{provably optimal} in the one-dimensional case and improves upon known bounds in higher dimensions. This rate, of the form $\sqrt{a} \log(1/a)$, contrasts with weaker results such as $a^{1/4}$ obtained in earlier works (e.g.,~\cite{Tra11}) and appears to be new even in classical settings such as Varadhan's original formulation.

From an engineering perspective, the significance of this study lies in its capability to perform sensitivity analysis based on the adjoint variable method for geometric features represented by the distance function, such as curvature and normal vectors \cite{HMTY,NYa,NYb}. 
This is particularly useful in the fields of topology optimization~\cite{ABBG23,OSY25,TY23} and structural analysis in measurement, because in topology optimization, shapes are represented by characteristic functions, and in structural analysis using X-rays, structures are represented by voxel data~\cite{SOS10}.

Using partial differential equations (PDEs) provides a powerful and flexible method for constructing distance functions. Among various approaches, we revisit and extend a classical elliptic formulation originally studied by Varadhan~\cite{V67}. To illustrate this idea, consider the equation:
\begin{equation}\label{eq:Varadhan}
   \begin{cases}
-a\Delta q_a+q_a = 0 &\text{ in } \Omega,\\
q_a=1&\text{ on } \partial \Omega,
\end{cases}
\end{equation}
where $\Omega$ is a domain in $\R^N$ with suitable regularity and $a>0$ is a small parameter. In the simple one-dimensional case $\Omega=(-h,h)$, where $h>0$, the solution is explicitly given by
\begin{equation}\label{eq:Varadhan2}
q_a(x)=\frac{\cosh(x/\sqrt{a})}{\cosh(h/\sqrt{a})}.
\end{equation}
A direct computation shows that
\[
-\sqrt{a}\log q_a(x) \to d(x,\partial\Omega), \qquad \text{as } a\to 0^+, \quad x\in \Omega,
\]
where $d(x,\partial\Omega):=\inf_{y\in \partial\Omega}|x-y|$ denotes the Euclidean distance to the boundary. In this example, we have $d(x,\partial \Omega) = h - |x|$. See Appendix~\ref{app:rate} for further details, including the case of higher-dimensional ball domains. Remarkably, this convergence is rigorously justified by Varadhan~\cite{V67} for general $N \geq 1$ and even with variable coefficients, where the limiting distance reflects the underlying Riemannian structure (see also~\cite{EI,B,P}).

It is worth noting that the distance function can also be obtained from other PDE-based formulations, such as the Hamilton--Jacobi and heat equations~\cite{O88,AJT02,AJT04,BM19,V67}. However, the present approach, based on elliptic equations with added source terms, offers novel advantages in terms of rate of convergence and practical applicability.

\eat{
The main aim of this paper is to propose a new
method for constructing the \emph{distance function} 
and to establish the convergence rate,
which is particularly optimum 
in the one-dimensional case.
From an engineering perspective, the significance of this study lies in its capability to perform sensitivity analysis based on the adjoint variable method for geometric features represented by the distance function, such as curvature and normal vectors. This is particularly useful in the fields of topology optimization \cite{ABBG23,TY23} and structural analysis in measurement because, in topology optimization, shapes are represented by characteristic functions, and in structural analysis using X-rays, structures are represented by voxel data \cite{SOS10}.

Using Partial Differential Equations (PDEs) provides a useful method for constructing the distance function.  
As an example, let us consider the following equation\/{\rm :}
\begin{equation}\label{eq:Varadhan}
   \begin{cases}
-a\Delta q_a+q_a = 0 &\text{ in } \Omega,\\
q_a=1&\text{ on } \partial \Omega,
\end{cases}
\end{equation}
where $\Omega$ is a domain of $\R^N$ with suitable regularity and $a>0$ is constant. 
Let us briefly describe the case $\Omega=(-h,h)$ in one dimension, where $h>0$. 
The solution to the above equation is explicitly written as 
\begin{equation}\label{eq:Varadhan2}
q_a(x)=\frac{\cosh(x/\sqrt{a})}{\cosh(h/\sqrt{a})}.
\end{equation}
It is easy to check that  
 $$
-\sqrt{a}\log q_a(x)\to  d(x,\partial\Omega),\quad \text{ $x\in \Omega$,\quad as $a\to 0^+$}, 
 $$
 where $d(x,\partial\Omega):=\inf_{y\in \partial\Omega}|x-y|$ is the distance function with respect to $\partial \Omega$. 
 In the above example, it is simply given by $d(x,\partial \Omega) = h-|x|$. See Appendix \ref{app:rate} for further details, including the case of higher dimensional ball domains. 
Surprisingly, the above convergence is rigorously justified by Varadhan \cite{V67} in general dimensions $N\ge 1$ and even with variable coefficients, in which the corresponding Riemann metric appears in the limit (see also \cite{EI,B,P}).         
Note that the distance function can also be extracted from other PDEs, e.g.~the Hamilton--Jacobi equation and the heat equation, see \cite{O88,AJT02,AJT04,BM19,V67}.  
}

\subsection{Problem setup}

We propose a new PDE method for constructing the distance function based on the following linear elliptic equation with inhomogeneous Dirichlet boundary condition\/{\rm :}  
\begin{equation}\label{eq:sd}
\begin{cases}
-a\Delta u_a+u_a = f &\text{ in } \Omega,\\
u_a=g&\text{ on } \partial \Omega,
\end{cases}
\end{equation}
where $a>0$ is a constant parameter. In case $f=0$ and $g=1$, this equation falls down to \eqref{eq:Varadhan}. 
Though we also use a linear elliptic equation, our view is rather different from the classical case; namely, we recognize $\Omega$ as a fixed design domain, and the ``domain of interest'' is defined as $[f=0]:=\{x \in \Omega: f(x)=0\}$ with the source term $f$. 
As a typical example of $f$, one can take the characteristic function with respect to $\omega\subset\Omega$ (i.e.,~$f=\chi_\omega$), which implies that we focus on the distance to the boundary of (the material) $\omega$ varying within the fixed domain $\Omega$, as in the structural optimization process \cite{A02}.  
The following example in one dimension will be helpful for grasping an asymptotic feature as $a\to0^+$.

\begin{figure}[b]
        \centering
   \begin{minipage}{0.4\hsize}  
   \begin{center}
    \includegraphics[keepaspectratio, scale=0.35]{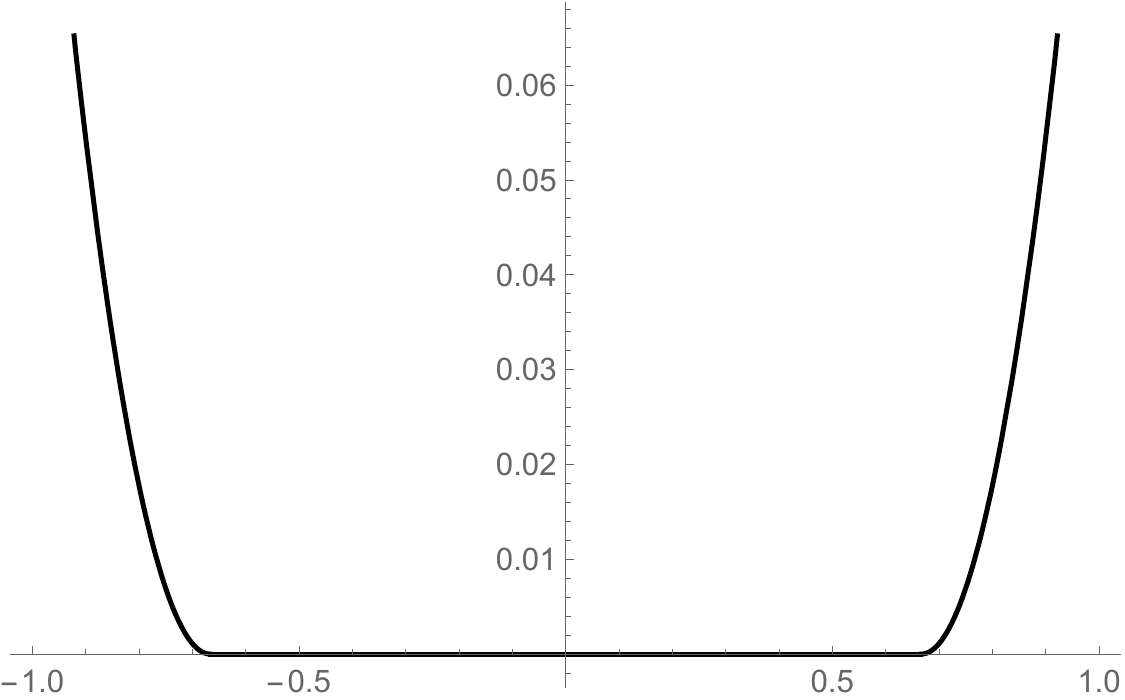}
\end{center}
\end{minipage} 
\hspace{15mm}
  \begin{minipage}{0.4\hsize} 
   \begin{center}
    \includegraphics[keepaspectratio, scale=0.35]{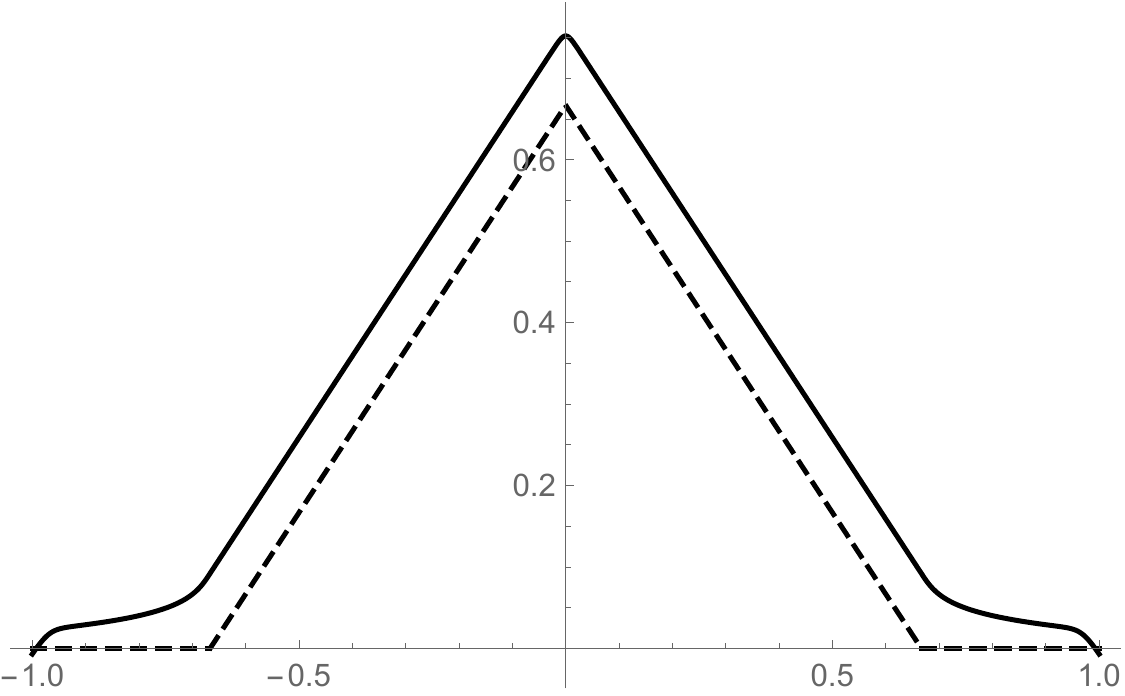}
\end{center}
\end{minipage}
   \caption{The graphs of $f(x)$ (left), $-\sqrt{a} \log u_a (x)$ (right, solid line) and $k-|x|$ (right, dashed line) in the setting of Example \ref{ex:1D} with $h=1$, $k=\frac2{3}$, $\alpha=1$, $a=0.0001$ and $\zeta=2$. }
\label{fig:1D}
\end{figure}

\begin{ex} \label{ex:1D}
Let $u_a$ be the solution to $-au_a'' + u_a =f$ on $(-h,h)$ and $u_a(-h)= u_a(h)=\alpha$, where $a,h,\alpha>0$ and 
\[
f(x) := \begin{cases} 0, & |x| < k, \\   (|x|-k)^\zeta,  & |x| \ge k, \end{cases}
\]
with $\zeta \ge 0$ and $0<k<h$. By direct calculations, we can see that
\[
u_a(x) = \gamma_a \cosh \frac{x}{\sqrt{a}}   - \frac1{\sqrt{a}} \int_0^x  f(y) \sinh \frac{x-y}{\sqrt{a}}\, \d y, \qquad x\ge0, 
\]
where 
\[
\gamma_a = \frac{\alpha + \frac1{\sqrt{a}} \int_0^h  f(y) \sinh \frac{h-y}{\sqrt{a}}\, \d y }{ \cosh \frac{h}{\sqrt{a}}}. 
\]
Then, as $a\to0^+$, 
\[
\gamma_a \sim \frac{1}{\sqrt{a}e^{\frac{h}{\sqrt{a}}}}  \int_k^h e^{\frac{h-y}{\sqrt{a}}}  \left(y-k\right)^\zeta \, \d y =  \frac{a^{\frac{\zeta+1}{2}}e^{\frac{h-k}{\sqrt{a}}}}{\sqrt{a}e^{\frac{h}{\sqrt{a}}}}  \int_0^{\frac{h-k}{\sqrt{a}}} e^{-z}  z^\zeta \, \d z \sim  \Gamma(1+\zeta) a^{\frac{\zeta}{2}} e^{- \frac{k}{\sqrt{a}}},  
\]
and hence, uniformly on $[0, k]$, 
\begin{align*}
\log u_a(x) 
&=  \log \gamma_a  + \log \cosh \frac{x}{\sqrt{a}}  \\
&=   \log \frac{\gamma_a}{\Gamma(1+\zeta) a^{\frac{\zeta}{2}} e^{- \frac{k}{\sqrt{a}}}} + \log \Gamma(1+\zeta) a^{\frac{\zeta}{2}} e^{- \frac{k}{\sqrt{a}}} +        \log \cosh \frac{x}{\sqrt{a}}  \\
& = \frac{1}{\sqrt{a}}\left( x  - k\right)   + \frac{\zeta}{2}\log a + O(1), \qquad a\to0^+.  
\end{align*}
By symmetry, this yields that, uniformly on $[-k,k]$, 
\[
- \sqrt{a} \log u_a(x) = \left(k-|x|\right)  + \frac{\zeta}{2}\sqrt{a}\log \frac1{a} + O(\sqrt{a}), \qquad a\to0^+.  
\]
See also Figure \ref{fig:1D}. 

\end{ex}

The function $k-|x|$ is exactly the distance of the point $x$ to the boundary of the support of $f$, not to the boundary of $\Omega$.  
The appearance of the distance function can be justified in a more general setting, which is the main result of the present paper.

\subsection{Main results}
Throughout this paper, the following assumptions for Equation \eqref{eq:sd} are employed. 

\vspace{4mm}
\noindent
{\bf Assumption.}   
\vspace{1mm}

\begin{enumerate}[label=\rm(\roman*),itemsep=0.5em,topsep=0.5em]

   \item\label{item:Omega} $\Omega$ is a bounded Lipschitz domain of $\R^N$ and satisfies an exterior sphere condition at each point of $\partial \Omega$.  
   
  \item\label{item:open}  $f\colon \Omega \to \R$ is measurable, bounded and nonnegative in $\Omega$, and the set $[f>0]:=\{x\in \Omega\colon f(x)>0\}$ is nonempty and open in $\Omega$.

    \item\label{item:away} The set  
$\AA:=\Omega\setminus \overline{[f>0]}$ is nonempty and satisfies $d(\AA, \partial\Omega)>0$, where $d(\AA,\partial \Omega):=\inf_{x\in \AA, y  \in \partial \Omega}|x-y|$. 

\item \label{item:C2'} The above set $\AA$ satisfies a \emph{uniform exterior sphere condition}, i.e., there exists $\e_0>0$ satisfying the following property: for any $\sigma\in\partial \AA$, there exists an open ball $B \subseteq \Omega \setminus \AA$ of radius $\e_0$ tangent to $\partial \AA$ at $\sigma$, i.e., $\sigma \in \partial \AA\cap \partial B$.

    \item\label{item:boundary} $g$ is a nonnegative function on $\partial\Omega$ and $g \in C(\partial \Omega)\cap H^{\frac1{2}}(\partial\Omega)$.   
\end{enumerate}

\begin{figure}[b]
        \centering
        \includegraphics[keepaspectratio, scale=0.3]{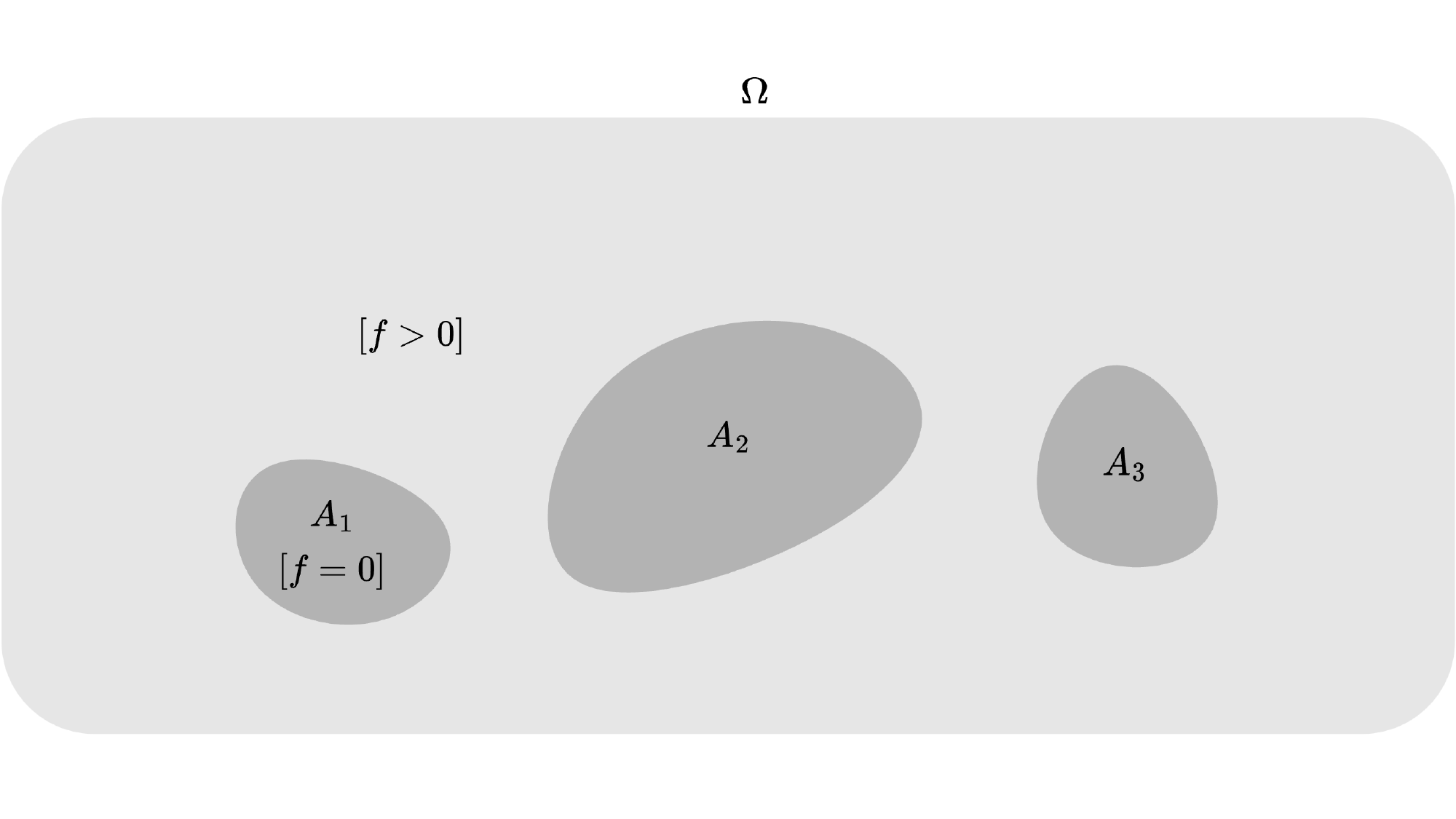}
   \caption{A typical configuration of $\Omega$ and $\AA =\AA_1 \cup \AA_2 \cup \AA_3$.}
\label{fig:concept1}
\end{figure}

\begin{rmk}
In Theorem \ref{C:main}, we also assume a uniform interior sphere condition, i.e., the condition ``$B\subseteq \Omega\setminus \AA$'' in \ref{item:C2'} is replaced with  $B\subseteq\AA$. In  \cite[Corollary 2]{LP20} it is shown that a bounded domain $\AA$ satisfies both uniform exterior and interior sphere conditions if and only if $\AA$ is a $C^{1,1}$-domain. 
\end{rmk}

\begin{rmk}
One could remove the exterior sphere condition for $\Omega$ from Assumption \ref{item:Omega} if one used \cite[Problem 6.3]{GT}. Actually, the exterior sphere condition will be needed only to use \cite[Theorem 6.13]{GT} later. According to \cite[Problem 6.3]{GT}, the exterior \emph{cone} condition, which follows from the Lipschitz boundary condition for $\Omega$, is enough to show \cite[Theorem 6.13]{GT}. However, since we could not find a proof of \cite[Problem 6.3]{GT} in the literature, we keep to assume the exterior sphere condition that is satisfied by practical examples such as balls and rectangles.   
\end{rmk}

As $\Omega$ is Lipschitz and $g \in H^{\frac{1}{2}}(\partial\Omega)$, by the surjectivity of the trace operator (see e.g.~\cite[Theorem 1.5.1.3]{Gri11}), there exists 
$\tilde{g}\in H^1(\Omega)$ such that $\tilde g |_{\partial \Omega} = g$. 
Then according to \cite[Theorem 8.3]{GT}
there exists a unique weak solution to \eqref{eq:sd}, i.e., a function $u_a\in H^1(\Omega)$ such that $u_a -\tilde g \in H^1_0(\Omega)$ and the following holds for all $\varphi\in H^1_0(\Omega)$: 
\begin{equation}\label{eq:sd_weak}
\int_\Omega a\nabla u_a(x)\cdot \nabla \varphi (x)\, \d x+ \int_{\Omega} u_a(x)\varphi(x)\, \d x=\int_\Omega f(x)\varphi(x)\, \d x. 
\end{equation}
Moreover, $u_a$ turns out to be in $W^{2,p}_{\rm loc}(\Omega)\cap C(\overline{\Omega})$ for any $1 \le p <+\infty$ by \cite[Theorem 9.30]{GT} and $u_a$ is positive on $\Omega$, see Subsection \ref{S:positivity}.

Then the main result reads as follows.  
\begin{thm}[Construction of distance function]\label{T:main} Under Assumptions \ref{item:Omega} -- \ref{item:boundary}, 
it holds that, as $a\to0^+$, 
$$
-\sqrt{a}\log u_a(x) \to d (x, \partial \AA)
\quad \text{ uniformly on } \overline{\AA}.   
$$   
\end{thm}

\begin{rmk}
  As the above result shows, to construct the distance function $d(\cdot,\partial \AA)$, there is arbitrariness for selecting $\Omega$ and $g$. In \cite{OSY25}, $\Omega$ was simply a rectangle and a simple boundary function $g\equiv1$ was selected. Nevertheless, since specializing to such simple examples does not simplify the proofs, we work under the most general assumptions on $\Omega$ and $g$.    
\end{rmk}

The convergence rate can also be obtained by carefully estimating the steps in the proof of Theorem \ref{T:main}. Let $B_\epsilon(y)$ stand for the open ball of radius $\e>0$ centered at $y\in \AA$ and $\vmean_\epsilon^y(H)$ be the volume mean of a function $H$ in the ball $B_\epsilon(y)$: 
$$
\vmean_\epsilon^y(H)
:=
\frac{1}{| B_\epsilon(y)|}
\int_{B_\epsilon(y)} H(x)\, \d x.
$$
In order to get a rate of convergence, we put a moderate assumption on the volume mean of $f$.

\begin{thm}[Rate of convergence]\label{T:rate} In addition to   Assumptions \ref{item:Omega} -- \ref{item:boundary}, 
we further assume that there exists $\zeta \ge0$ such that
\begin{equation} \label{eq:mean_f}
0<  \inf_{\substack{0< \epsilon <1\\ y \in \partial \AA}}\epsilon^{-\zeta p} \vmean_\epsilon^y(f^p) \le  \sup_{\substack{0< \epsilon <1\\ y \in \partial \AA}}\epsilon^{-\zeta p} \vmean_\epsilon^y(f^p) <\infty, \qquad  p\in\{1,2\}. 
\end{equation}
 Then for every $0<\tau <1/2$ there exists a constant $C>0$ such that
\begin{equation*}
\sup_{x\in \overline{\AA}} \left| -\sqrt{a} \log u_a(x) - d(x, \partial \AA)\right| \le Ca^{\frac1{2}-\tau} \quad \text{for all} \quad 0<a<\frac1{2}.  
\end{equation*}
Moreover, if $\zeta=0$ then the bound $C a^{\frac1{2}-\tau}$ can be replaced by $C \sqrt{a} \log \frac{1}{a}$ that is optimal at least in dimension one, see Example \ref{ex:1D}. 
\end{thm}
\begin{rmk}
    In case of $f\equiv0$ and $g\equiv1$, $u_a$ coincides with the solution 
$q_a$ to Equation \eqref{eq:Varadhan} and $\AA$ coincides with $\Omega$ as $[f>0]=\emptyset$. Although this case is excluded from our main result due to Assumption \ref{item:away}, our proof can be easily fixed to treat this case. More precisely, under the modified assumption 
\begin{enumerate}[label=\rm(\roman*'),topsep=0.7em]

   \item\label{item:Omega'} $\Omega$ is a bounded domain of $\R^N$ with a   uniform exterior sphere condition,   
 \end{enumerate}
we can prove the bound 
 \begin{equation}\label{eq:rate_Varadhan}
\sup_{x\in \overline{\Omega}} \left| -\sqrt{a} \log q_a(x) - d(x, \partial \Omega)\right| \le C\sqrt{a}\log\frac{1}{a} \quad \text{for all} \quad 0<a<\frac1{2},    
\end{equation} 
which seems a new result to the best of authors'  knowledge. 
In fact, the most difficult parts, Steps 2 and 3 in Subsection \ref{S:lower} and arguments in Section \ref{S:rate}, are no longer  needed, so that the proof is much simpler. This rate of convergence is optimal in dimensions $N\ge2$ from an explicit estimate in the case of ball domains; see Appendix \ref{app:rate}.  It is worth mentioning that Tran studied more general possibly nonlinear equations and obtained the weaker bound $C a^{\frac1{4}}$ \cite[Theorem 3.3]{Tra11}, see also Remark \ref{rmk:viscosity} below. 
\end{rmk}

\begin{rmk}
    
Condition \eqref{eq:mean_f} is technical and the exponent $\zeta$ only appears implicitly in the constant $C$. Whether this condition can be removed or not is unclear. The number $\zeta$  represents the  smoothness (or the speed of damping) of $f$ near $\partial \AA$. This condition is satisfied by many practical examples, e.g.\ the characteristic function $f=\chi_{\Omega \setminus \overline{\AA}}$ $(\zeta=0)$ and functions of the form $f(x)=\max\{0, p(x)\}$ where $p$ is a smooth function  (typically $\zeta \ge 1$). On the other hand, functions $f(x)$ that go to zero  exponentially fast as $d(x,\partial \AA)\to0$ provide counterexamples (``$\zeta=+\infty$''). See Examples \ref{ex:f1} -- \ref{ex:f3} for further details.  
\end{rmk}

In the special case where the source function $f$ is a characteristic function, we can also capture a distance function on a part of the complement of $\AA$. Also, in this case, we can get the optimal convergence rate.

\begin{thm}[Construction of (local) signed distance function]\label{C:main}
In addition to Assumptions \ref{item:Omega} -- \ref{item:boundary}, we also assume a uniform interior sphere condition for $\AA$.    We select $f:= C^* \chi_{\Omega \setminus \overline{\AA}}$ with any constant $C^* \ge \sup_{x\in \partial\Omega}g(x)$ and define  
 $U_a\in H^1(\Omega)$ as 
\begin{align}\label{eq:sdist}
    U_a(x):=
\begin{cases}
\sqrt{a}\log u_a(x),\quad &x \in \overline{\AA}, \\ 
-\sqrt{a}\log ( C^* -u_a(x)),\quad &x\in \Omega \setminus \overline{\AA} 
\end{cases}
\end{align}
and the signed distance function 
\[
\hat{d}(x,\partial \AA):= 
\begin{cases}
-d(x,\partial \AA), & x \in \overline{\AA}, \\
d(x, \partial \AA), & x \in \Omega \setminus \overline{\AA}. 
\end{cases}
\]
Then there exists $C>0$ such that 
\begin{equation}\label{eq:rate_signed_distance}
\sup_{x \in \Omega^*} |U_a(x) - \hat{d}(x,\partial \AA)| \le C \sqrt{a} \log \frac1{a} \quad\text{for all} \quad 0< a< \frac{1}{2}, 
\end{equation}
where  $\Omega^\ast:=\{ x\in \Omega\colon d(x,\partial \AA)\le d(x,\partial\Omega)\}$. 
\end{thm}

\begin{rmk}\label{rmk:viscosity}
    Closely related problems were studied in the context of the method of vanishing viscosity for viscosity solutions to Hamilton--Jacobi equations. First, applying the Cole-Hopf transform $p_a(x) = -\sqrt{a} \log u_a(x)$ to our equation \eqref{eq:sd} yields  
\begin{equation*}
\begin{cases}
-\sqrt{a}\Delta p_a+|\nabla p_a|^2-1 =-f\exp(\frac{p_a}{\sqrt{a}})& \text{ in } \Omega, \\
p_a = -\sqrt{a}\log g & \text{ on }  \partial \Omega. 
\end{cases}
\end{equation*}
Since $f$ vanishes in $A$, the above equation implies   
 \begin{equation}\label{eq:eikonal}
-\sqrt{a}\Delta p_a+|\nabla p_a|^2-1 =0 \quad \text{ in } A. 
\end{equation}
Equations more general than \eqref{eq:eikonal} were treated in \cite{FS86} and \cite{Tra11} with zero boundary condition on $\partial A$. In \cite[Theorem 5.1]{FS86} an asymptotic expansion was obtained. 
In \cite[Theorem 3.3]{Tra11}, the rate of convergence $a^{\frac1{4}}$ was obtained as mentioned above. In our problem, the  values of $p_a$ on $\partial A$ are not a priori given. 
The major part of our proof is devoted to the estimate of these boundary values. 
The same rate of convergence $a^{\frac1{4}}$ has been obtained in \cite[Proposition 6.1]{L} (see also the remark after the proposition) when the domain is whole space $\R^N$.
For optimal convergence rates in the vanishing viscosity limit for quadratic (and more generally uniformly convex) Hamilton--Jacobi equations, see \cite{CD}.
\end{rmk}

\subsection{Structure and technical contribution of this paper}
This paper is composed of four sections.
Section \ref{S:main} is devoted to proving Theorem \ref{T:main}. 
The basic strategy of the proof is similar to \cite{V67}: we use explicit comparison functions 
for $u_a$ to obtain upper and lower bounds of $-\sqrt{a}\log u_a$. The upper estimate of $u_a$ is obtained rather directly from the comparison technique.  
However, the lower estimate of $u_a$ is much more involved than that in \cite{V67} and requires a substantially new idea. In \cite{V67}, the values of the solution on the interface (i.e.,~$\partial\Omega$) are constant and are independent of $a>0$ due to the inhomogeneous Dirichlet boundary condition, see \eqref{eq:Varadhan}. On the other hand, the values of $u_a$ on the interface $\partial A$ of our target depend on $a$. 
This is the most delicate and complex issue in this paper. 
To overcome this difficulty, we shall elaborate on the lower bound of the solution on the interface, i.e., $\beta_a:=\inf_{x\in \partial A}u_a(x)$, and then prove $\sqrt{a}\log \beta_a\to 0$ as $a\to 0^+$, which plays a crucial role in the proof
(the uniform lower bound $\inf_{x\in \Omega}u_a(x)$ is generally too small and seems to be useless for the proof). 
To prove this, we shall establish a bound for $\sqrt{a} \log\beta_a$ in terms of the volume mean of $u_a$ on small balls by comparing $u_a$ with the explicit solution to \eqref{eq:sd} (see Lemma \ref{P:mean-value}) in the case where $f=0$ and $\Omega$ is a ball. Then we find a bound for the volume mean making use of a global energy bound for $u_a$, which concludes the proof of    $\sqrt{a}\log \beta_a\to 0$. 

Section \ref{S:rate} deals with the rate of convergence to the distance function.
This is achieved by showing a fine rate of convergence  $\sqrt{a}\log \beta_a\to 0$. For this purpose, we refine the global energy estimate of $u_a$ to a local one.  

Section \ref{sec:signed} provides a sketch of the proof of Theorem \ref{C:main}. 
It is quite similar to that of Theorem \ref{T:main}.

\section{Proof of Theorem  \ref{T:main}}\label{S:main}


\subsection{Preparation: surface mean and volume mean}\label{S:claim1}

Let $\eta>0$, $y\in \R^N$, $F\in C(\partial B_\eta(y))$ and $\smean_\eta(F)$ stand for the surface mean of $F$ on $\partial B_\eta(y)$, i.e.,
$$
\smean_\eta^y(F)
:=
\begin{cases}
\displaystyle \frac{1}{|\partial B_\eta(y)|}
\int_{\partial B_\eta(y)} F(x)\, \d \sigma_x, & \text{if}~N\ge2, 
\\[8mm]
\displaystyle\frac{F(y+\eta) + F(y-\eta)}{2}, & \text{if}~N=1. 
\end{cases}
$$
For a function $H \in C(D)$ with $\partial B_\eta(y) \subseteq D \subseteq \R^N$, we simply denote by $\smean_\eta^y(H)$ the surface mean of the restriction $H|_{\partial B_\eta(y)}$ on $\partial B_\eta(y).$

Recall that the volume mean and surface mean are related as follows.  
\begin{lem}\label{L:means} Let $H\in C(\overline{B_\eta(y)})$. Then  
\begin{equation*}
    \vmean_\eta^y(H) 
    = \frac{N}{\eta^N} \int_0^\eta  r^{N-1}   \smean_r^y(H)\, \d r.   
\end{equation*}
\end{lem}
\begin{proof}
For $N=1$ the formula can be directly checked. We assume $N\ge2$. In the polar coordinate $x = r \omega + y~ (r >0,~\omega \in \partial B_1(0))$, 
it holds that $\d x = |\partial B_r(0)|\,\d r\d\omega =  |\partial B_1(0)|r^{N-1}\,\d r\d\omega$ (see Remark \ref{rem:polar} below).
This yields
\begin{align*}
\int_{B_\eta(y)} H(x) \,\d x 
    &=  |\partial B_1(0)|\int_0^\eta r^{N-1}\left(  \int_{\partial B_1(0)} H(r \omega + y) \, \d\omega\right) \d r \displaybreak[0]
\\
    &= |\partial B_1(0)| \int_0^\eta r^{N-1} \smean_r^y(H) \,\d r.  
\end{align*}
The desired formula follows from the relation $ |\partial B_1(0)|=N|B_1(0)|$. 
\end{proof}
\begin{rmk}\label{rem:polar}
  A parametric representation of polar coordinate is given e.g.\ by $\omega_i=\left[\prod_{k=1}^{i-1} \sin \theta_k\right] \cos \theta_i~(1\le i \le N-1)$, $\omega_{N}=\prod_{k=1}^{N-1} \sin \theta_k$, $0\le \theta_i \le \pi~(1\le i \le N-2)$, $0\le \theta_{N-1} \le 2\pi$. Then the volume element is $\d x = \d r \d \sigma^r$, where $ \d\sigma^r:= r^{N-1} \prod_{i=1}^{N-1}\left[ \sin^{N-1-i} \theta_i \,\d\theta_i\right]$ is the surface area element on $\partial B_r(0)$ and $\d \omega := \d \sigma^1/|\partial B_1(0)|$  is the uniform distribution on the unit sphere.
\end{rmk}

\begin{lem}\label{P:mean-value} 
Let $H\in C^2(B_\eta(y)) \cap C\big( \overline{B_\eta(y)}\big)$ be a classical solution to the modified Helmholtz equation 
\begin{equation*}
a\Delta H = H \quad\text{ in }\quad B_\eta(y). 
\end{equation*}
Then it holds that 
\[
H(y) = 
\begin{cases}
\displaystyle \frac{\smean_\eta^y(H) \int_0^\pi \sin^{N-2}\theta \,\d\theta}{\int_0^\pi \cosh\left(\frac{\eta \cos \theta}{  \sqrt{a}} \right)\sin^{N-2}\theta \,\d\theta} , & \text{if}~ N\ge2, \displaybreak[0]
\\[8mm]
\displaystyle \frac{\smean_\eta^y(H)}{\cosh\left(\frac{\eta}{\sqrt{a}}\right)}, & \text{if}~N=1. 
\end{cases}
\]
\end{lem}
\begin{proof} The case $N=1$ follows by directly solving the equation. The case $N\ge2$ is exactly the second formula in \cite[(2.5)]{Kuz}. For the reader's convenience, we give a proof in Appendix \ref{app:mean}. 
\end{proof}

\subsection{Preparation: positivity and uniform bound for solutions.} \label{S:positivity}

The solution $u_a$ of \eqref{eq:sd} satisfies
\begin{align}\label{eq:minmax}
0<u_a(x)\le M:=\max\left\{\sup_{x\in \Omega} f(x),  \sup_{x\in \partial\Omega} g(x)\right\}\quad \text{ for all } x\in \Omega.
\end{align}
Indeed, let $v\in C^2(\Omega)\cap C(\overline{\Omega})$ be a classical solution to the following equation\/{\rm :}
\begin{equation*}
\begin{cases}
a\Delta v = v &\text{ in } \Omega,\\
v=g&\text{ on } \partial \Omega.
\end{cases}
\end{equation*}
The existence of $v$ is guaranteed by \cite[Theorem 6.13]{GT}. 
Setting $U=v-u_a$, we see that $U\in H^1_0(\Omega)$ satisfies
\begin{equation*}
\begin{cases}
-a\Delta U+ U=-f
&\text{ in } \Omega,\\
U=0&\text{ on } \partial \Omega.
\end{cases}
\end{equation*}
Then testing it by $U_+:=\max\{U,0\}\in H^1_0(\Omega)$, one can derive that
\begin{align*}
    \int_{\Omega}|U_+(x)|^2\, \d x
    &= \int_{\Omega}U(x)U_+(x)\, \d x \displaybreak[0]
\\
    &\le
    \int_{\Omega}U(x)U_+(x)\, \d x + \int_{\Omega} a\nabla U(x)\cdot \nabla U_+(x)\, \d x \displaybreak[0]
\\
    &=
    -\int_\Omega f(x)
    U_+(x)\, \d x\le 0, 
\end{align*}
which yields $v\le u_a$ in $\Omega$. By the  weak maximum principle (see e.g.\ \cite[Corollary 3.2]{GT}), $v$ is nonnegative and so $0\le v\le u_a$ in $\Omega$.  
Moreover, \cite[Theorem 9.6]{GT} applied to $-u_a$ implies that either $u_a > 0$ on $\Omega$ or $u_a$ is constant on $\Omega$, the latter of which never happens since $f$ is nonzero. Therefore, $u_a >0$ on $\Omega$. The positivity guarantees that we can consider the function $\log u_a$. 
 
For the boundedness from above,  let $u_a^\star := M-u_a$. Then 
\begin{equation*}
\begin{cases}
-a\Delta u_a^\star+u_a^\star = M-f &\text{ in } \Omega,\\
u_a^\star=M-g&\text{ on } \partial \Omega. 
\end{cases}
\end{equation*} Since $M-f \ge0$ on $\Omega$ and $M-g \ge0$ on $\partial\Omega$, the previous arguments show $u_a^\star \ge0$, which is exactly the upper bound of \eqref{eq:minmax}.


\subsection{Proof of  Theorem \ref{T:main}} \label{S:lower}

Our general strategy is to split 
the proof of Theorem \ref{T:main} into two parts: there exist $m_a \in \R$, $a >0$, such that  
\begin{align}\label{claim1}
\inf_{x\in \overline{\AA}}[-\sqrt{a}\log u_a(x)- d(x,\partial \AA)]
\ge  m_a \qquad \text{and} \qquad \liminf_{a\to0^+} m_a \ge 0, 
\end{align}
and there exist $M_a \in \R$, $a >0$, such that 
\begin{align}\label{claim2}
\sup_{x\in \overline{\AA}}[-\sqrt{a}\log u_a(x)-d(x,\partial \AA)]
\le  M_a \qquad \text{and} \qquad \limsup_{a\to0^+} M_a \le 0.  
\end{align}

\subsubsection{Proof of lower bound}
We begin with the lower limit \eqref{claim1}, which is much easier. 
Let us consider the following equation\/{\rm :}
\begin{equation}\label{eq:dens}
\begin{cases}
a \Delta \hat{u}_a = \hat{u}_a &\text{ in } \AA,\\
\hat{u}_a=u_a&\text{ on } \partial \AA.
\end{cases}
\end{equation}
By \cite[Theorem 6.13]{GT}, Equation \eqref{eq:dens} has a unique classical solution $\hat{u}_a \in C^2(\AA)\cap C(\overline{\AA})$. In addition, \cite[Theorem 9.5]{GT} implies $u_a = \hat{u}_a$ on $\overline{\AA}$.  

We take an arbitrary $y\in \AA$ and fix it. Let $\e:=d(y,\partial \AA)$ and 
let $w=w_y\in C^2(B_\e(y))\cap C(\overline{B_\e(y)})$ be a unique classical solution to 
\begin{equation*}
\begin{cases}
a\Delta w = w &\text{ in } B_\e(y),\\
w=M&\text{ on } \partial B_\e(y),
\end{cases}
\end{equation*}
where $M$ is the uniform upper bound of $u_a$ defined in \eqref{eq:minmax}.  
Since $u_a\le w$ on $\partial B_\e(y)$, we see by the comparison principle that $u_a\le w$ on $B_\e(y)$.
We can deduce from Lemma \ref{P:mean-value} that
\begin{align*}
u_a(y)\le w(y)
&\le \pi  M  \left(\int_a^{\arccos(1-a)} \frac{\exp((\e\cos\theta)/\sqrt{a})}{2}\sin^{N-2}\theta\, \d\theta \right)^{-1}\\
&\le \pi  M  \left(\frac{\exp(\e(1-a)/\sqrt{a})}{2}\right)^{-1}
\left(\int_a^{\arccos(1-a)} \sin^{N-2}\theta\, \d\theta \right)^{-1}\\
&\le  
\frac{2\pi  M  \exp(-\e(1-a)/\sqrt{a})}{[\arccos(1-a)-a]\sin^{N-2}a} \quad \text{ if $N\ge 2$}
\end{align*}
and 
\[ 
u_a(y)\le
w(y)= \frac{2 M}{\exp(\e/\sqrt{a})+\exp(-\e/\sqrt{a})}\le 2 M e^{-\frac{\e}{\sqrt{a}}}\ \text{ if $N=1$}.
\] 
 Therefore, we have
\begin{align*}
\sqrt{a}\log u_a(y)\
\le 
\begin{cases}
-d(y,\partial \AA)(1-a)+\sqrt{a}\log C_{N,a}M
\quad &\text{ if }\ N\ge 2,\\
-d(y,\partial \AA
)+\sqrt{a}\log 2M&\text{ if }\ N=1, 
\end{cases}
\end{align*}
where $C_{N,a}:=2\pi/[(\arccos(1-a)-a) \sin^{N-2}a] \asymp (1/a)^{N-\frac{3}{2}}$  as $a\to 0^+$  
by noting that $\arccos(1-a) \sim \sqrt{2a}$, $ a\to0^+$.
Thus we obtain  
\begin{align} 
-\sqrt{a}\log u_a(y)-d(y,\partial \AA)
\ge 
\begin{cases}
-a\,{\rm diam}(\AA)-\sqrt{a}\log C_{N,a} M & 
\text{ if }\ N\ge 2,\\
-\sqrt{a}\log 2M&\text{ if }\ N=1. 
\end{cases}   
\label{eq:lower_bound_ua}
\end{align}
Since this lower bound is uniform over $y\in \AA$ and  $\sqrt{a}\log u_a$ is continuous on $\overline{\AA}$, we obtain \eqref{claim1} as desired. 

\begin{rmk}
 Because of the constant boundary value, $w$ has the following explicit formula in $B_\epsilon(y)$: 
\begin{align*}
w(x)=
\begin{cases}
M \frac{\displaystyle\int_0^\pi \cosh((|x-y|\cos \theta)/\sqrt{a})\sin^{N-2}\theta\, \d \theta}{\displaystyle\int_0^\pi \cosh((\e\cos \theta)/\sqrt{a})\sin^{N-2}\theta\, \d \theta}
&\text{ for }\ N\ge 2, 
\vspace{3mm}\\
\displaystyle
M\frac{ \cosh(|x-y|/\sqrt{a})}{ \cosh(\e/\sqrt{a})}
&\text{ for }\ N=1. 
\end{cases}
\end{align*}
For our purpose,  the value of $w$ at the center $y$, i.e., Lemma \ref{P:mean-value}, was sufficient as shown above. 
\end{rmk}

\subsubsection{Proof of upper bound}\label{S:claim2} 
We next establish the upper bound \eqref{claim2}, where 
the quantity
\[
\beta_a :=\inf_{y\in\partial \AA}u_a(y) >0  
\]
turns out to be of great help.

\vspace{1mm}
\noindent
\textbf{Step 1.} We establish a bound for $u_a$ of the form
\begin{align}
&-\sqrt{a}\log u_a(y)-d(y,\partial \AA)
\le  \label{eq:upper_bound_ua}  \\
&\qquad\qquad  -\sqrt{a}\log \beta_a   
+
\begin{cases}
a\,{\rm diam}(\AA)+\sqrt{a}(1+a)+\sqrt{a}\log \tilde{C}_{N,a} 
&\text{ if }\ N\ge 2,\\
\sqrt{a}&\text{ if }\ N=1  
\end{cases} 
\notag 
\end{align}
for all $y \in \overline{\AA}$ and $a>0$, where $\tilde{C}_{N,a}$ is a positive constant independent of $y$ and such that $\tilde{C}_{N,a} \asymp (1/a)^{N-\frac{3}{2}}$, $a\to 0^+$.   

\begin{figure}[b]
        \centering
        \includegraphics[keepaspectratio, scale=0.14]{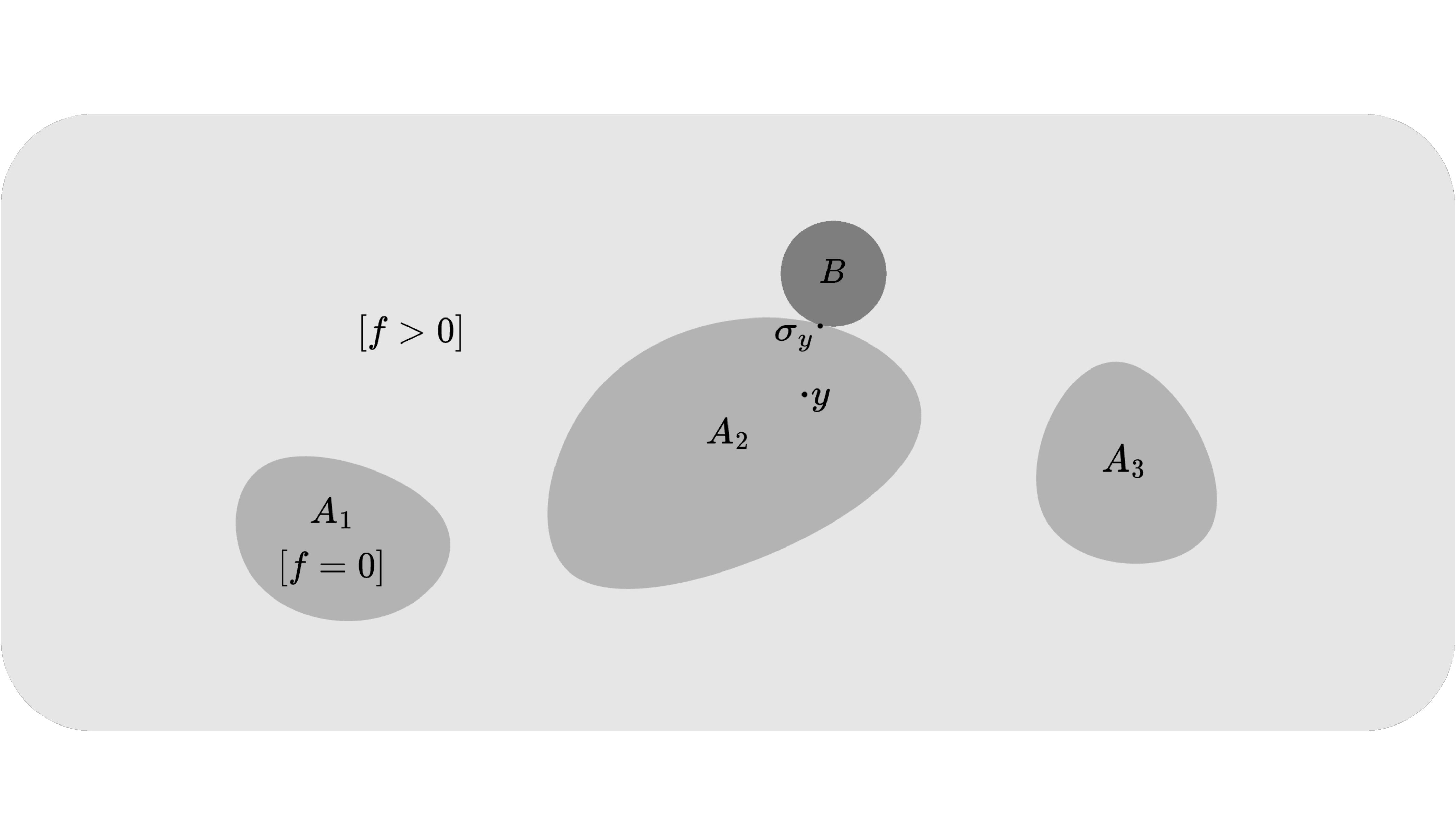}
   \caption{A typical configuration of $\Omega$, $\AA =\AA_1 \cup \AA_2 \cup \AA_3$ and $B=B_{\sqrt{a}}(\tilde{y}).$ }
\label{fig:concept2}
\end{figure}

Towards a proof, we take an arbitrary $y \in \AA$ and fix it. Let $\sigma = \sigma_{y}\in \partial \AA$ be a point such that $|y-\sigma| =d(y, \partial \AA)$. For each $0<a \le \e_0^2$, there exists a ball $B_{\sqrt{a}}(\tilde{y}) \subseteq \Omega\setminus \AA$ tangent to $\partial \AA$ at $\sigma$ according to Assumption \ref{item:C2'}, see Figure \ref{fig:concept2}. (The point $\tilde y$ will depend on $a$ and $y$.) Note then that 
\begin{align}\label{eq:triangle_ineq}
    |y-\tilde{y}|\le d(y,\partial \AA)+\sqrt{a}.      
\end{align}
Let $z$ be the unique classical solution to 
\begin{equation*}
\begin{cases}
a\Delta z = z &\hspace{-15mm}\text{ in } \R^N\setminus B_{\sqrt{a}}(\tilde{y}),\\
z=\beta_a&\hspace{-15mm}\text{ on } \partial B_{\sqrt{a}}(\tilde{y}),\\
\displaystyle \lim_{|x|\to +{\infty}}z(x)=0. &
\end{cases}
\end{equation*}
Recall that the solution $z$ can be written as
\begin{align} \label{eq:zineeq}
z(x)=
\begin{cases}
\frac{\displaystyle
\beta_a\int_0^\infty \exp(- (|x-\tilde{y}|\cosh t)/\sqrt{a})\sinh^{N-2}t\, \d t}{\displaystyle
\int_0^\infty \exp(- \cosh t)\sinh^{N-2}t\, \d t}
&\text{ for }\ N\ge 2, 
\vspace{3mm}\\
\displaystyle \beta_a\exp(-|x-\tilde{y}|/\sqrt{a})
&\text{ for }\ N=1. 
\end{cases}
\end{align}
Since $u_a \ge \beta_a\ge z$ on $\partial \AA$ by \eqref{eq:zineeq} (recall from arguments around \eqref{eq:dens} that $u_a|_{\overline{\AA}} \in C^2(\AA)\cap C(\overline{\AA})$ is a classical solution to \eqref{eq:dens}), the comparison principle \cite[Theorem 3.3]{GT} ensures that $u_a \ge z$ in $ \AA$.  
Noting that $\acosh(1+a) \ge  \asinh(a)$ for all $a>0$, we infer from \eqref{eq:triangle_ineq} and \eqref{eq:zineeq} that, if $N\ge2$,  
\begin{align*}
u_a(y)
&\ge z(y)\\
&\ge
 \frac{\beta_a}{c_N} 
\int_{\asinh(a)}^{\acosh(1+a)}\exp(-(|y-\tilde{y}|\cosh t)/\sqrt{a}\,) \sinh^{N-2}t\, \d t\\
&\ge
\frac{\beta_a}{c_N} a^{N-2}[\acosh(1+a)-\asinh(a)] 
 \exp\left(-\frac{\big[d(y,\partial \AA) +\sqrt{a}\,\,\big](1+a)}{\sqrt{a}}\right),
\end{align*}
where $c_N>0$ is the denominator of $z$ in \eqref{eq:zineeq}, and if $N=1$, 
\[ 
u_a(y)\ge z(y)
\ge \beta_a
\exp\left(-\frac{\big[d(y,\partial \AA) +\sqrt{a}\,\,\big]}{\sqrt{a}}\right).
\]
Passing to the logarithm yields 
\begin{align*}
\sqrt{a}\log u_a(y)
\ge 
\sqrt{a}\log \beta_a  
+
\begin{cases}
-[d(y,\partial { \AA}
)+\sqrt{a}\,](1+a) - \sqrt{a}\log \tilde{C}_{N,a}\quad &\text{ if } N\ge 2,\\
-\left[d(y,\partial { \AA})
+\sqrt{a}\,\,\right] \quad &\text{ if } N=1, 
\end{cases}
\nonumber
\end{align*}
where $\tilde{C}_{N,a}:= c_N/[a^{N-2}(\acosh(1+a)-\asinh(a))]$, and hence the desired inequality \eqref{eq:upper_bound_ua}. The asymptotic behavior $\tilde{C}_{N,a} \asymp (1/a)^{N-\frac{3}{2}}$, $a\to 0^+$, holds since $\acosh(1+a) \sim \sqrt{2a}$ and $ \asinh(a)\sim a$ as $a\to0^+$. Note that, since this is uniform on $y\in \AA$, it also holds on $\overline{\AA}$ by the continuity of $\sqrt{a}\log u_a$. 

In view of \eqref{eq:upper_bound_ua}, it remains to prove that 
\begin{equation}\label{eq:ca}
\liminf_{a\to 0^+}\sqrt{a}\log \beta_a\ge 0,
\end{equation}
which is achieved in Steps 2 and 3 below. Note that this inequality is equivalent to  $\lim_{a\to0^+}\sqrt{a}\log\beta_a=0$ because $\beta_a$ is uniformly bounded from above, cf.~\eqref{eq:minmax}.

\vspace{1mm}
\noindent
\textbf{Step 2.} For each $a>0$, we choose a point $y_a \in \partial \AA$ such that $u_a(y_a)=\beta_a$, which exists by the continuity of $u_a$. 
We prove that, for any $0<\epsilon<d(\partial \AA, \partial \Omega):= \inf\{|x-y|:(x,y)\in\partial A \times \partial \Omega\}$ and $a>0$,  
\begin{align} 
\sqrt{a}\log \beta_a
\ge
-\epsilon+
\sqrt{a}\log \vmean_{\epsilon}^{y_a}({u_a}).  \label{eq:logC_a}
\end{align}

For each  $0<\eta<  d(\partial \AA, \partial \Omega)$, let $u_a^\eta$ be a unique classical solution to
\begin{equation*}
\begin{cases}
a\Delta u_a^\eta = u_a^\eta &\text{ in } B_\eta(y_a),\\
u_a^\eta={u}_a&\text{ on } \partial B_\eta(y_a).
\end{cases}
\end{equation*}
We see by Lemma \ref{P:mean-value} that  
\begin{align}\label{eq:uaeta}
u_a^\eta(y_a) \ge \frac{\smean_\eta^{y_a}({u}_a)}{ \cosh(\eta/\sqrt{a})} 
\end{align}
in any dimensions $N$. 
On the other hand, ${u}_a\ge u_a^\eta$ holds in $B_\eta(y_a)$ by the same argument as in the proof of $ u_a \ge v$ in $\Omega$ (see Subsection \ref{S:positivity}), which then yields  
\begin{align}
    u_a(y_a) 
&= 
\frac{N}{\epsilon^N}\int_0^\epsilon \eta^{N-1} u_a(y_a)\, \d\eta  \notag \\
&\ge    
\frac{N}{\epsilon^N}\int_0^\epsilon \eta^{N-1}u_a^\eta (y_a)\, \d\eta
\quad \text{ for }\quad 0<\epsilon< d(\partial \AA, \partial \Omega).  
\label{eq:keyineq}
\end{align}
Combining  \eqref{eq:keyineq} with \eqref{eq:uaeta} and Lemma \ref{L:means}, we infer that, for any $0<\epsilon< d(\partial \AA, \partial \Omega)$,
\begin{align*}
u_a(y_a) 
&\ge 
\frac{N}{\epsilon^N\cosh(\epsilon/\sqrt{a})}\int_0^\epsilon \eta^{N-1}\smean_\eta^{y_a}(u_a) \, \d\eta  =
\frac{\vmean_\epsilon^{y_a}(u_a)}{\cosh(\epsilon/\sqrt{a})} \ge
e^{-\frac{\epsilon}{\sqrt{a}}} \vmean_\epsilon^{y_a}(u_a), 
\end{align*}
so that  \eqref{eq:logC_a} holds.

\vspace{1mm}
\noindent
\textbf{Step 3.} 
We show that, for any $0<\epsilon<d(\partial \AA, \partial \Omega)$, 
\begin{equation}\label{eq:Mconv}
    \lim_{a\to 0^+}\sup_{y \in \partial \AA}|\vmean_{\epsilon}^y({u}_a)-
    \vmean_{\epsilon}^y({f})|=0.
\end{equation}
Before going to the proof, we observe that  \eqref{eq:logC_a} and \eqref{eq:Mconv} imply \eqref{eq:ca}.  Indeed, by Assumption \ref{item:open},   $\vmean_\epsilon^y(f)>0$ for all $y \in \partial \AA$. Since  $y\mapsto \vmean_\epsilon^y(f)$ is continuous on $\partial \AA$, the number  $\alpha:=\inf_{y\in\partial \AA}\vmean_\epsilon^y(f)$ is positive. 
 From \eqref{eq:Mconv}, $|\vmean_{\epsilon}^{y_a}({u_a}) -\vmean_{\epsilon}^{y_a}(f)| <\alpha/2$ for sufficiently small $a>0$ and hence $\vmean_{\epsilon}^{y_a}({u_a}) \ge \alpha/2$. 
 This yields $\lim_{a\to0^+}\sqrt{a}\log \vmean_{\epsilon}^{y_a}({u_a}) =0$, and hence, by \eqref{eq:logC_a}, $\liminf_{a\to0^+}\sqrt{a}\log \beta_a \ge -\epsilon$. Since $0< \epsilon < d(\AA,\partial \Omega)$ was arbitrary, this completes the proof of \eqref{eq:ca} and hence of Theorem \ref{T:main}. 

Equation \eqref{eq:Mconv} can be proven as follows. 
To begin with, testing the weak form \eqref{eq:sd_weak}  with $\varphi:= {u}_a-\tilde{g}\in H^1_0(\Omega)$, we observe that  
\begin{align}
&
\int_\Omega a\nabla {u}_a(x)\cdot \nabla ({u}_a (x)-\tilde{g}(x))\, \d x
+
\int_\Omega {u}_a(x)({u}_a (x)- \tilde{g} (x))\, \d x \label{eq:energy1}  \\
&\qquad =
\int_\Omega 
f(x)({u}_a(x)-\tilde{g} (x))\, \d x  \notag   \le
\frac{1}{2}\|{f}\|_{L^2(\Omega)}^2+\frac{1}{2}\|{u}_a\|_{L^2(\Omega)}^2
+
\|{f}\|_{L^2(\Omega)}\|\tilde{g}\|_{L^2(\Omega)}.  \notag 
\end{align}
On the other hand,  
\begin{align*}
&\text{the left hand side of \eqref{eq:energy1}} \\
&\qquad \ge 
\frac{1}{2}\|\sqrt{a}\nabla {u}_a\|_{L^2(\Omega)}^2
-
\frac{1}{2}\|\sqrt{a}\nabla \tilde{g} \|_{L^2(\Omega)}^2
+
\|{u}_a\|_{L^2(\Omega)}^2
-
\left(
\frac{1}{2\Lambda}\|{u}_a\|_{L^2(\Omega)}^2+\frac{\Lambda}{2}\|\tilde{g}\|_{L^2(\Omega)}^2
\right)
\end{align*}
for any $\Lambda>0$. By setting $\Lambda=2$, the global energy bound 
\begin{equation}\label{eq:energy}
\frac{1}{2}\|\sqrt{a}\nabla {u}_a\|_{L^2(\Omega)}^2+\frac{1}{4}\|{u}_a\|_{L^2(\Omega)}^2\le C
\end{equation}
holds for some $C>0$ independent of $0<a<1$.  

We fix $y \in \partial \AA$ and a number $\delta$ such that $0< \delta < \epsilon$. 
In the weak form \eqref{eq:sd_weak}, we select $\varphi_{\epsilon,\delta}^y := \chi_{\epsilon}^y \ast \xi_\delta$, where $ \chi_{\epsilon}^y$ is the indicator function of the ball $B_{\epsilon}(y)$, $\xi_\delta(x) := \delta^{-N} \xi(x/\delta)$,  and $\xi\colon \R^N\to [0,\infty)$ is a  $C^1(\R^N)$ function supported on the closed unit ball $\overline{B_1(0)}$ such that $\int_{\R^N} \xi =1$.

By the weak form
\begin{equation*}
\int_\Omega u_a \varphi_{\epsilon,\delta}^y  = \int_\Omega f \varphi_{\epsilon,\delta}^y  - a\int_\Omega \nabla u_a \cdot \nabla \varphi_{\epsilon,\delta}^y, 
\end{equation*}
we can get the following inequality: 
\begin{align}
& \epsilon^N |B_1(0)| |\vmean_{\epsilon}^y(u_a) - \vmean_\epsilon^y(f) | 
=  \left| \int_\Omega u_a \chi_{\epsilon}^y  - \int_\Omega f \chi_{\epsilon}^y \right|    \label{eq:mean_u_a_f}\\
&\qquad \le \left| \int_\Omega u_a \chi_{\epsilon}^y   -  \int_\Omega u_a \varphi_{\epsilon,\delta}^y  \right| 
+  \left| \int_\Omega f \varphi_{\epsilon,\delta}^y  - a\int_\Omega \nabla u_a \cdot \nabla \varphi_{\epsilon,\delta}^y - \int_\Omega f \chi_{\epsilon}^y \right|   \notag \\
&\qquad \le  \underbrace{\left| \int_\Omega u_a \chi_{\epsilon}^y   -  \int_\Omega u_a \varphi_{\epsilon,\delta}^y  \right|}_{=:I_1}     
+  \underbrace{\left| \int_\Omega f \varphi_{\epsilon,\delta}^y-\int_\Omega f \chi_{\epsilon}^y   \right| }_{=:I_2}
+ \underbrace{\left|a\int_\Omega \nabla u_a \cdot \nabla \varphi_{\epsilon,\delta}^y \right|}_{=:I_3}.    \notag
\end{align}
We estimate the last three terms below. To lighten the notation, we drop the superscripts $y$ and use the same symbol $C$ for possibly different positive constants that are independent of $y \in \partial \AA$.

Estimates on $I_1$ and $I_2$ follow from that of $\|\varphi_{\epsilon,\delta} -\chi_\epsilon\|_2$. 
Here and henceforth, $\|\cdot\|_p$ denotes $\|\cdot\|_{L^{p}(\Omega)}$, when no confusion can arise. 
In the expression
\[
\varphi_{\epsilon,\delta}(x) - \chi_\epsilon(x) =\int_{|z|<\delta} [\chi_\epsilon(x-z) - \chi_\epsilon(x)] \xi_\delta(z)\, \d z, 
\]
we observe that 
\[
\sup_{|z|<\delta}|\chi_\epsilon(x-z) - \chi_\epsilon(x)| \le 
\begin{cases} 0,  & |x-y| < \epsilon -\delta \quad \text{or} \quad    |x-y|>\epsilon+\delta, \\
 1,&  \epsilon -\delta \le |x-y| \le \epsilon+\delta 
\end{cases}
\]
so that 
\begin{equation} \label{eq:approx_indicator}
\|\varphi_{\epsilon,\delta} - \chi_\epsilon\|_{2} \le  \sqrt{|B_1(0)|} \sqrt{(\epsilon+\delta)^N - (\epsilon-\delta)^N} \le C (\delta \epsilon^{N-1})^\frac1{2}.   
\end{equation}
By the Schwarz inequality and uniform boundedness of $u_a$, we have 
\begin{equation*}
I_1 \le  C\|\varphi_{\epsilon,\delta} - \chi_\epsilon\|_{2} \le C(\delta \epsilon^{N-1})^{\frac1{2}}
\end{equation*}
and similarly 
\begin{equation*}
I_2\le   C(\delta \epsilon^{N-1})^{\frac1{2}}. 
\end{equation*}

For $I_3$, Young's inequality implies
\begin{equation}\label{eq:Young1} 
\|\nabla \varphi_{\epsilon,\delta}\|_2 = \|\chi_\epsilon \ast \nabla \xi_{\delta}\|_2 \le \|\chi_\epsilon\|_2 \|\nabla \xi_{\delta}\|_1 \le C\epsilon^{\frac{N}{2}} \delta^{-1}.   
\end{equation}
 Combining \eqref{eq:energy} and \eqref{eq:Young1} yields 
\begin{equation*} 
I_3 \le   \sqrt{a} \|\sqrt{a} \nabla u_a\|_2 \| \nabla \varphi_{\epsilon,\delta} \|_{2} \le C  \sqrt{a}\epsilon^{\frac{N}{2}} \delta^{-1} .
\end{equation*}

Altogether we deduce
\begin{equation*}
 |\vmean_{\epsilon}^y(u_a) - \vmean_\epsilon^y(f) | \le  \frac{I_1+I_2 + I_3}{|B_1(0)| \epsilon^N}\le  C \delta^{\frac1{2}}\epsilon^{-\frac{N+1}{2}} + C \sqrt{a} \delta^{-1} \epsilon^{-\frac{N}{2}}. 
\end{equation*}
Note that the constant $C$ is independent of $y \in \partial \AA$,  $0<a<1$ and $0<\delta < \epsilon <d(\partial \AA,\partial \Omega)$. 
Taking $\sup_{y \in \partial \AA}$,   $\limsup_{a\to0^+}$ and then $\lim_{\delta\to0^+}$ yields \eqref{eq:Mconv} as desired.

\section{Proof of Theorem \ref{T:rate}}\label{S:rate}
Rate of convergence can be obtained by careful estimates of each approximation in the proof of Theorem \ref{T:main}. The lower bound 
\begin{align*}
\inf_{x\in \overline{\AA}} \left[-\sqrt{a} \log u_a(x) - d(x, \partial \AA)\right] \ge -C\sqrt{a} \log \frac1{a}     
\end{align*}  actually follows from \eqref{eq:lower_bound_ua}. This lower bound is better than $- C a^{\frac1{2}-\tau}$ and does not require condition  \eqref{eq:mean_f}.
It remains to prove the upper bound \begin{align}\label{eq:rate_upper}
\sup_{x\in \overline{\AA}} \left[-\sqrt{a} \log u_a(x) - d(x, \partial \AA)\right] \le  Ca^{\frac{1}{2}-\tau}.
\end{align}
One can see from arguments in Subsection \ref{S:claim2}, especially from Equations \eqref{eq:logC_a} and \eqref{eq:upper_bound_ua}, that the main problem is to find a good lower bound for $\sqrt{a}\log \beta_a$ as a function of $a$, which follows from a good lower bound for $\vmean_{\epsilon}^y(u_a)$ with suitably chosen function $\e=\e(a)$. 

\subsection{Preparation: local energy bounds}

The strategy is to give careful estimates on the integrals $I_1,I_2,I_3$ in \eqref{eq:mean_u_a_f}. 
The following local version of energy bound \eqref{eq:energy} plays a crucial role.   

\begin{lem}\label{lem:local_energy}   For any $a, \eta >0$ and $y \in \Omega$ such that $\overline{B_{3\eta}(y)} \subseteq \Omega$, the inequality  
 \begin{equation}\label{eq:ineq_eta}
a\|\nabla u_a  \|_{L^2(B_{\eta}(y))}^2  +  \| u_a\|_{L^2(B_{\eta}(y))}^2 \le \|f\|_{L^2(B_{3\eta}(y))}^2 + C_N  a \eta^{-2}  \|u_a\|_{L^2(B_{3\eta}(y))}^2 
 \end{equation}
 holds, where $C_N>0$ is a constant depending only on the dimension $N$. 
\end{lem}
\begin{proof}
We take $\phi =\phi_\eta \in C_{\rm c}^1(\R^N)$ such that $0\le \phi \le 1$ on $\R^N$, $\text{supp}(\phi) \subseteq  \overline{B_{3 \eta}(y)}$ and $\phi=1$ on $B_{\eta}(y)$. Taking the test function $\varphi:=\phi^2 u_a$ in the weak form yields  
\begin{equation}\label{eq:energy2}
 a\int_\Omega \nabla u_a \cdot \nabla [\phi^2 u_a]    +  \int_\Omega  \phi^2u_a^2  = \int_\Omega  \phi^2f u_a. 
\end{equation}
We observe that   
\begin{equation*}
\text{the right hand side of \eqref{eq:energy2}} \le \frac1{2}\int_\Omega \phi^2 (f^2 + u_a^2) =   \frac1{2}\|\phi f\|_{2}^2 + \frac{1}{2}\|\phi u_a\|_2^2. 
\end{equation*}  
On the other hand, using the inequality $2\alpha \cdot \beta \ge  - \frac{|\alpha|^2}{2} -2 |\beta|^2~(\alpha,\beta \in \R^N)$, we have that 
\begin{align*}
\text{the left hand side of \eqref{eq:energy2}} 
&=  a\int_\Omega \phi^2    |\nabla u_a|^2 +    2a\int_\Omega (\phi  \nabla u_a) \cdot ( u_a \nabla \phi ) + \|\phi u_a\|_2^2 \\
&\ge a\| \phi \nabla u_a  \|_2^2 -  a \left( \frac{1}{2} \|\phi \nabla u_a\|_2^2 + 2 \| u_a\nabla \phi\|_2^2 \right) +  \|\phi u_a\|_2^2 \\
&=  \frac{a}{2}\| \phi \nabla u_a  \|_2^2 - 2 a  \|u_a \nabla \phi\|_2^2 +  \|\phi u_a\|_2^2, 
\end{align*}
so that 
\begin{equation}\label{eq:nabla_u}
 a\| \phi \nabla u_a  \|_2^2  +  \|\phi u_a\|_2^2  \le  \|\phi f\|_{2}^2 + 4 a  \|u_a \nabla \phi\|_2^2. 
 \end{equation}
 To estimate $\|\nabla \phi\|_\infty$,  we select $\phi := \chi_{2\eta} \ast \xi_\eta$, where $\chi_{2\eta}$ is the indicator function of the ball $B_{2\eta}(y)$, $\xi_\eta(x) := \eta^{-N} \xi(x/\eta)$,  and $\xi\colon \R^N\to [0,\infty)$ is a $C^1(\R^N)$ function supported on the closed unit ball $\overline{B_1(0)}$ such that $\int_{\R^N} \xi =1$. Then we obtain 
\begin{equation}\label{eq:nabla_phi}
 \| \nabla \phi  \|_\infty \le  \| \chi_{2\eta}\|_1 \|\nabla \xi_\eta\|_\infty  =    |B_1(0)| (2\eta)^{N}  \|\nabla \xi\|_{\infty}\eta^{-N-1}. 
 \end{equation}
Combining \eqref{eq:nabla_u} and \eqref{eq:nabla_phi} and properties of $\phi$ yields the desired inequality. 
\end{proof}

\subsection{Proof of the upper bound \eqref{eq:rate_upper} in Theorem \ref{T:rate}}  \label{S:upper_bound}

For simplicity, numerous positive constants appearing below are all denoted by $C$. These constants might depend on $N,\tau,\zeta$, but not on $a$ or $y$.  
The notation in Step 3 of Subsection \ref{S:claim2}  is inherited. We fix $y \in \partial \AA$.  Let $0 < \tau' < \tau <1/2$ be arbitrary and 
\begin{equation*}  
\epsilon = a^{\frac1{2} -\tau} \qquad \text{and} \qquad \delta = a^{\frac1{2}-\tau'}. 
\end{equation*}
Note that $\epsilon> \delta$ for  $0<a<1$.

\vspace{1mm}
\noindent
\textbf{Estimating local energy.}
 In \eqref{eq:ineq_eta} note that condition   \eqref{eq:mean_f} implies $\|f\|_{L^2(B_{3\eta}(y))}^2 \le C\eta^{N+2\zeta}$. We specialize to the parameter $\eta =3^k\epsilon$, $k =0,1,2,\dots$, where $\epsilon = a^{\frac1{2}-\tau}$.  Then, neglecting the term $a\|\nabla u_a  \|_{L^2(B_{\eta}(y))}^2$,  we get 
 \begin{equation}\label{eq:ineq_rec}
  \| u_a\|_{L^2(B_{3^k \epsilon}(y))}^2 \le  C (3^{(N+2\zeta)k} \epsilon^{N+2\zeta} + 3^{-2k} a^{2\tau}  \|u_a\|_{L^2(B_{3^{k+1}\epsilon}(y))}^2).   \tag{Ineq${}_k$}
 \end{equation}
Recursively substituting (Ineq$_{k}$) into the right hand side of (Ineq$_{0}$) amounts to 
 \begin{equation*}
  \| u_a\|_{L^2(B_{\epsilon}(y))}^2 \le C (\epsilon^{N+2\zeta} + a^{2\tau n}  \|u_a\|_{L^2(B_{3^{n}\epsilon}(y))}^2) 
 \end{equation*}
 for all $0<a<1$ and $n\in\N$ such that $\overline{B_{3^n \epsilon}(y)} \subseteq \Omega$.  
Choosing $n$ to be the minimal integer such that $2 \tau n \ge (\frac1{2}-\tau)(N+2\zeta)$ and  recalling that $\|u_a\|_2$ is uniformly bounded, we get the bound 
 \begin{equation}\label{eq:ineq_ep1}
  \| u_a\|_{L^2(B_{\epsilon}(y))}^2 \le   C \epsilon^{N+2\zeta} 
 \end{equation}
as long as $\overline{B_{3^n \epsilon}(y)} \subseteq \Omega$.  
 Substituting  \eqref{eq:ineq_ep1} into the right hand side of \eqref{eq:ineq_eta} with $\eta=\epsilon=a^{\frac1{2}-\tau}$, we obtain a bound for the local energy
 \begin{equation}\label{eq:nabla_u4}
 \|\sqrt{a} \nabla u_a  \|_{L^2(B_{\epsilon}(y))}^2  +  \| u_a\|_{L^2(B_{\epsilon}(y))}^2  \le  C \epsilon^{N+2\zeta},  
 \end{equation}
 where $C>0$ is independent of $a>0$ and $y \in \partial \AA$.

\vspace{1mm}
\noindent
{\bf Estimating $I_1, I_2$ and $I_3$.}  Note that the integrands in $I_1,I_2,I_3$ are supported on the ball $B_{2\epsilon}(y)$ since the test function vanishes outside it.  
By the Schwarz inequality, \eqref{eq:approx_indicator} and the local energy estimate, we have 
\begin{equation}\label{eq:varphi_chi3}
I_1 \le  \|\varphi_{\epsilon,\delta} - \chi_\epsilon\|_{2} \|u_a\|_{L^2(B_{2\epsilon}(y))} \le C(\delta \epsilon^{N-1})^{\frac1{2}}  \epsilon^{\frac{N+2\zeta}{2}} =   C \epsilon^{N+\zeta} (\delta/\epsilon)^{\frac1{2}} 
\end{equation}
and similarly 
\begin{equation}\label{eq:varphi_chi4}
I_2\le    C \epsilon^{N+\zeta} (\delta/\epsilon)^{\frac1{2}}. 
\end{equation}

Concerning $I_3$, 
 combining \eqref{eq:nabla_u4} and \eqref{eq:Young1} yields 
\begin{equation} \label{eq:nabla_varphi}
I_3 \le   \sqrt{a} \|\sqrt{a} \nabla u_a\|_{L^2(B_{2\epsilon}(y))} \| \nabla \varphi_{\epsilon,\delta} \|_{2} \le C \epsilon^{N+\zeta} \sqrt{a} \delta^{-1} .
\end{equation}

\vspace{2mm}
\noindent
{\bf Lower bound of $\vmean_\epsilon^y(u_a)$.}
Combining \eqref{eq:mean_u_a_f},  \eqref{eq:varphi_chi3},  \eqref{eq:varphi_chi4} and \eqref{eq:nabla_varphi}  together we deduce
\begin{equation*}
 |\vmean_{\epsilon}^y(u_a) - \vmean_\epsilon^y(f) | \le  \frac{I_1+I_2 + I_3}{|B_1(0)| \epsilon^N}\le   Ca^{\min\{\frac{\tau -\tau'}{2}, \tau'\}} \epsilon^\zeta. 
\end{equation*}
Since condition  \eqref{eq:mean_f} exactly means $ \vmean_\epsilon^y(f) \asymp \epsilon^\zeta$, $\epsilon\to0^+$,  the right hand side is much smaller than $ \vmean_\epsilon^y(f)$, so that 
\begin{align*}
 \vmean_{\epsilon}^y(u_a) \ge c \vmean_\epsilon^y(f) \ge c' \epsilon^{\zeta}
\end{align*}
for some $c,c' >0$, which, together with \eqref{eq:logC_a}, further entails  
\[
\sqrt{a}\log \beta_a \ge  \sqrt{a}\log \vmean_\epsilon^y(u_a) -\epsilon \ge  - C a^{\frac1{2}-\tau}. 
\]
By this inequality and \eqref{eq:upper_bound_ua} we obtain \eqref{eq:rate_upper} as desired. 

\vspace{1mm}
\noindent
\textbf{The case $\zeta=0$.}
Although the proofs above are all valid for the case $\zeta=0$ too, some changes of proofs provide a better upper bound. The necessary changes are noted below. The parameters $\epsilon, \delta$ are replaced with 
\[
\epsilon = \sqrt{a} \log \frac{1}{a} \qquad \text{and} \qquad \delta=  \sqrt{a} \left(\log \frac{1}{a}\right)^{\frac1{2}}. 
\]
In \eqref{eq:nabla_u}, we put $\eta=\epsilon$,  use the uniform boundedness of $u_a$ and Young's inequality,  to proceed as 
\[
4 a  \|u_a \nabla \phi\|_2^2 \le Ca \|\nabla \phi\|_2^2 \le Ca  ( \|\nabla \xi_\eta\|_1 \|\chi_{2\eta}\|_2 )^2 \le C a \eta^{-2} \eta^N \le C \epsilon^N.  
\]
This yields the local energy estimate (without the tricky iterative substitution argument)
\begin{equation*}
 a\|\nabla u_a  \|_{L^2(B_{\epsilon}(y))}^2  +  \| u_a\|_{L^2(B_{\epsilon}(y))}^2  \le  C \epsilon^{N}. 
 \end{equation*}
 The remaining arguments are quite similar.

\subsection{Supplementary examples}

We present examples of source functions $f$ that satisfy/do not satisfy condition \eqref{eq:mean_f}.

\begin{ex} \label{ex:f1}
In addition to Assumptions \ref{item:Omega} -- \ref{item:boundary}, we assume that $\AA$ has a $C^1$-boundary. Let $f= c \chi_{\Omega \setminus \overline{\AA}}$ be the characteristic function multiplied by a constant $c>0$. We can see that 
\[
 \lim_{\epsilon\to0^+}\vmean_\epsilon^y(f^p) =\frac{c^p}{2},  \qquad p\in \{1,2\}, 
\]
uniformly on $y \in \partial \AA$. This 
shows that $f=c \chi_{\Omega \setminus \overline{\AA}}$ satisfies condition \eqref{eq:mean_f} with $\zeta=0$. 
\end{ex}

\begin{ex}\label{ex:f2} Let $\AA := B_1(0) \subset \R^N$ and $f (x) := [(|x|^2-1) \vee 0]^\zeta$ with $\zeta\ge0$ fixed.  For each $y \in \partial B_1(0)$ and $p\in\{1,2\}$ we have
\begin{align*}
\vmean_\epsilon^y (f^p) 
&= \frac{1}{|B_\epsilon(y)|} \int_{B_\epsilon(y)} f(x) \, \d x  
= \frac{1}{\epsilon^N |B_1(0)|} \int_{|x-y | < \epsilon,~ |x|>1} (|x|^2-1)^{p\zeta} \, \d x \\
&=  \frac{1}{\epsilon^N |B_1(0)|}  \int_{0}^\epsilon \, \d r  \int_{S_{r}} r^{N-1} (| y+ r \omega |^2-1)^{p\zeta}\, \d\sigma(\omega)\\
&=   \frac{1}{\epsilon^N |B_1(0)|}  \int_{0}^\epsilon  r^{N+p\zeta -1} \underbrace{ \left(  \int_{S_{r}}   (2  \omega \cdot y  +r)^{p\zeta} \, \d \sigma(\omega) \right)}_{=:\, F(r)} \d r, 
\end{align*}
where $S_{r}$ is the surface $\{\omega \in \partial B_1(0) : |r\omega+y|>1\}$ and $\d\sigma$ is the (unnormalized) surface measure.  Since the surface $S_{r}$ contains the hemisphere $S_+^y:=\{\omega \in \partial B_1(0): \omega\cdot y >0 \}$, we have for $0<r<\e<1$
\[
c_{p\zeta} := \int_{S_+^y} (2\omega \cdot y  )^{p\zeta} \, \d\sigma(\omega) \le  \int_{S_r} (2\omega \cdot y  )^{p\zeta} \, \d\sigma(\omega) \le F(r) \le 3^{p\zeta}|\partial B_1(0)|. 
\]
As $c_{p\zeta}$ is a positive constant independent of $y$ and $\e$, the inequality $\tilde c_{p\zeta} \epsilon^{p\zeta} \le \vmean_\epsilon^y (f^p) \le \tilde C_{p\zeta} \epsilon^{p\zeta}$ holds for some constants $0< \tilde c_{p\zeta} < \tilde C_{p\zeta}$ independent of $y \in \partial B_1(0)$ and $\epsilon\in(0,1)$. Thus condition \eqref{eq:mean_f} is fulfilled. 
\end{ex}

\begin{ex}\label{ex:f3} Let $\AA := B_1(0) \subset \R^N$, $q>0$ and 
\[
f (x) := 
\begin{cases}\exp[-(|x|^2-1)^{-q}], & |x|>1, \\
0, & |x| \le 1. 
\end{cases}
\]
Analogously to the calculations in Example \ref{ex:f2} we obtain 
for each $y \in \partial B_1(0)$ and $\e\in(0,1)$
\begin{align*}
\vmean_\epsilon^y (f) 
=   \frac{1}{\epsilon^N |B_1(0)|}  \int_{0}^\epsilon  r^{N -1} \underbrace{ \left(  \int_{S_{r}}  \exp\left(-\frac{1}{(2r\omega\cdot y + r^2)^q}\right)  \d \sigma(\omega) \right)}_{=:\, G(r)} \d r. 
\end{align*} 
The function $G$ can be estimated as 
\begin{align*}
G(r) &\le   \int_{\partial B_1(0)}  \exp\left(-\frac{1}{(3r)^q}\right)  \d \sigma(\omega) = |\partial B_1(0)|\exp\left(-\frac{1}{(3r)^q}\right)  \\
&\le  |\partial B_1(0)|\exp\left(-\frac{1}{2(3\e)^q}\right)\exp\left(-\frac{1}{2(3r)^q}\right)
\end{align*}
and so 
\[
\vmean_\epsilon^y (f) \le \frac{|\partial B_1(0)|}{|B_1(0)|} \varepsilon^{-N} \exp\left(-\frac{1}{2(3\e)^q}\right) \underbrace{\int_{1/\e}^\infty t^{-N-1}e^{-t^q/(2\cdot 3^q)}\,\d t.}_{=\,o(1),~\e\to0^+} 
\]
This inequality ensures that $\e^{-\zeta}\vmean_\e^y(f)\to0$ as $\e\to0^+$ for any fixed finite nonnegative exponent $\zeta$, so that $f$ does not satisfy condition  \eqref{eq:mean_f}. 
\end{ex}

\section{Proof of Theorem \ref{C:main}} \label{sec:signed}

Firstly, note that the desired inequality \eqref{eq:rate_signed_distance}  on $\overline{\AA}$ was already established in Theorem \ref{T:rate}. It remains to work on $\Omega^* \setminus \overline{\AA}$. 
It is easy to check that the functions $u_a^* := C^* - u_a$ and $f^* := C^* - f = C^* \chi_{\overline{\AA}}$ satisfy the PDE
\[
\begin{cases}
-a\Delta u_a^* + u_a^*= f^* & \quad \text{ on }\quad \Omega, \\
u_a^* = C^\ast - g &    \quad \text{ on }\quad \partial\Omega. 
\end{cases}
\]
Arguments very analogous to the proof of Theorems \ref{T:main} and \ref{T:rate} work for this PDE, which are outlined below.  One needs to pay attention to that $\AA^* := \Omega \setminus \text{supp}(f^*)$ does not satisfy Assumption \ref{item:away} and this is why we need to focus on the smaller domain $\Omega^*$.  
\begin{itemize}
\item It holds from \eqref{eq:minmax} that $0< u_a^* \le C^*$ on $\Omega$. 
\item Take a ball $B_\epsilon(y)$ with any $y \in \Omega^* \setminus \overline{\AA}$ and $\epsilon := d(y, \partial \AA)$ and a unique solution to the equation
 \begin{equation*}
\begin{cases}
a\Delta v_a^* =v_a^* &\text{ in } B_\e(y),\\
v_a^* =C^* &\text{ on } \partial B_\e(y).
\end{cases}
\end{equation*}
Using the explicit form of $v_a^*$ and the comparison $u_a^*(y) \le v_a^*(y)$, 
we can prove 
\begin{align*}
-\sqrt{a}\log u_a^*(y)- d(y,\partial \AA) 
&\ge 
\begin{cases}
-a\,{\rm diam}(\AA)-\sqrt{a}\log C_{N,a} C^* & 
\text{ if }\ N\ge 2,\\
-\sqrt{a}\log 2C^*&\text{ if }\ N=1 
\end{cases}   \\
&\ge - C \sqrt{a} \log\frac1{a} 
\end{align*}
for some $C>0$ depending only on $C^*, \text{diam}(\AA), N$. 
In the above arguments, it is important to take $y$ from $ \Omega^* \setminus \overline{\AA}$, not from $ \Omega \setminus \overline{\AA}$, because for $y\in \Omega \setminus \overline{\AA}$ the ball $B_\e(y)$ might not lie inside $\Omega$. 

\item Step 1 is very similar. For $y \in \Omega\setminus \overline{\AA}$ we can take $\sigma \in \partial \AA$ such that $|y-\sigma| = d (y, \partial \AA)$ and a ball $B_{\sqrt{a}}(\tilde y)$, now contained in $\AA$, and tangent to $\partial \AA$ at $\sigma$. Then we can prove 
\begin{align}
&\sup_{y \in \Omega \setminus \overline{\AA}}[-\sqrt{a}\log u_a^*(y)-d(y,\partial \AA)] 
\le   \notag \\
&\qquad  -\sqrt{a}\log \beta_a^*   
+
\begin{cases}
a\,{\rm diam}(\AA)+\sqrt{a}(1+a)+\sqrt{a}\log \tilde{C}_{N,a} 
&\text{ if }\ N\ge 2,\\
\sqrt{a}&\text{ if }\ N=1.  
\end{cases} 
\notag 
\end{align}
as desired. Note that here we do not need to work on the smaller domain $\Omega^*\setminus \overline{\AA}$.

\item The arguments in Step 2 and Step 3 are available with obvious changes: we can prove for any $y \in \Omega$ and $\eta>0$ with $B_\eta(y)\subseteq \Omega$ 
\[
\lim_{a\to0^+} \vmean_\eta^y(u_a^*) = \vmean_\eta^y(f^*)   ~~~\left( = \frac{C^*}{2} \right),  \quad \text{uniformly over $y \in \partial A$}, 
\]
and then prove 
\[
\liminf_{a\to0^+} \beta_a^* \ge0, 
\]
where $\beta_a^* := \inf_{y \in \partial \AA} u_a^*(y)$.

\item We can prove $\sqrt{a}\log \beta_a^* \ge -c \sqrt{a} \log \frac1{a}$, $0< a< 1/2$, for some constant $c>0$ following exactly the same lines of the proof of Theorem \ref{T:rate} (the case $\zeta=0$). Note that we do not need the iterative self-substitutions of \eqref{eq:ineq_rec}, see the last part of Subsection \ref{S:upper_bound}. 
\end{itemize}

\appendix 

\section{Proof of Lemma \ref{P:mean-value}} \label{app:mean} 

The following proof follows the lines of \cite{Kuz} but some parts are changed. Let $\mu := \frac1{\sqrt{a}}$. 
Note that Lemma \ref{L:means} reads 
\begin{equation}\label{eq:polar2}
\int_{B_r(y)}   H(x)\,\d x = N |B_1(0)| \int_0^r t^{N-1} \smean_t^y(H) \,\d t, \qquad 0< r <\eta. 
\end{equation}
We apply the Laplacian $\Delta_y$ to \eqref{eq:polar2}. Using the uniform distribution $\d \omega$ on the unit sphere $\partial B_1(0)$ and $\Delta H = \mu^2 H$, we have 
\[
\Delta_y \smean_r^y(H) = \Delta_y \int_{\partial B_1(0)} H(y + r\omega)\, \d\omega = \int_{\partial B_1(0)} \Delta H(y + r\omega)\, \d\omega = \mu^2 \smean_r^y(H),  
\]
so that 
\begin{equation}\label{eq:RHS}
\Delta_y(\text{the right hand side of \eqref{eq:polar2}}) 
=  N |B_1(0)|\mu^2 \int_0^r t^{N-1} \smean_t^y(H) \,\d t. 
\end{equation}
On the other hand, 
\begin{align} \label{eq:LHS}
&\Delta_y(\text{the left hand side of \eqref{eq:polar2}})  \\
&\qquad=  \Delta_y \int_{B_r(0)}   H(y+z)\,\d z 
= \int_{B_r(0)}   \Delta H(y+z)\,\d z \notag \\
&\qquad=  \int_{\partial B_r(0)}   \nabla H(y+z) \cdot \frac{z}{r}\,\d\sigma_z 
= |\partial B_r(0)|\int_{\partial B_1(0)}   \nabla H(y+r\omega) \cdot \omega\,\d\omega  \notag \\
&\qquad= N r^{N-1}|B_1(0)| \int_{\partial B_1(0)} \frac{\partial}{\partial r}  H(y+r\omega) \,\d\omega = N r^{N-1}|B_1(0)|\frac{\partial}{\partial r} \smean_r^y(H). \notag  
\end{align}
Equating \eqref{eq:RHS} and \eqref{eq:LHS} and differentiating them with respect to $r$ yield the ODE 
\begin{equation}\label{eq:ODE_smean}
    \partial_r^2 \smean_r^y(H)  + (N-1)r^{-1} \partial_r \smean_r^y(H) - \mu^2 \smean_r^y(H) =0. 
\end{equation}

Let $\nu:= \frac{N-2}{2}$ and $F(r) : = (\mu r)^\nu\smean_{r/\mu}^y(H)$, $0< r < \mu \eta$. Then \eqref{eq:ODE_smean} will turn into 
\begin{equation}\label{eq:Bessel}
  r^2  F''(r) + r F'(r) -(r^2 + \nu^2)F(r)=0, \qquad 0< r< \mu\eta.     
\end{equation}
This second order linear differential equation is known to have two linearly independent solutions $I_\nu(r)$ and $K_\nu(r)$ known as the modified Bessel functions of first and second kinds, respectively \cite[9.6.1]{AS}. Since a general solution to \eqref{eq:Bessel} is a linear combination of $I_\nu(r)$ and $K_\nu(r)$, there are constants $c_1, c_2 \in \R$ such that 
\[
F(r) = c_1 I_\nu(r) + c_2 K_\nu(r). 
\]
We can determine $c_1$ and $c_2$ from the asymptotics of the functions as $r\to0^+$. It is well known (see \cite[9.6.7 -- 9.6.9]{AS}) that 
\begin{align*}
    I_\nu(r) \sim \left(\frac{r}{2}\right)^\nu \frac1{\Gamma(1+\nu)}, \qquad r\to0^+, \\
    K_\nu(r)\to +\infty, \qquad r\to0^+.
\end{align*}
On the other hand, the definition of $F$ yields  
\[
F(r) =  (\mu r)^\nu(H(y)+o(1)), \qquad r\to0^+. 
\]
Combining these asymptotics we find that 
\[
c_1 = H(y) \Gamma(1+\nu) (2\mu)^\nu\qquad \text{and}\qquad c_2=0.  
\]
Taking the limit $r\uparrow \mu \eta$ in the relation $F(r) = c_1 I_\nu (r)$ yields 
\begin{equation}\label{eq:h}
H(y) = \frac{(\mu \eta/2)^\nu \smean_\eta^y(H)}{\Gamma(1+\nu) I_\nu(\mu \eta)}. 
\end{equation}

Finally, combined with the integral representation (see \cite[9.6.18]{AS}) 
\begin{align*}
I_\nu(r) 
&= \frac{(r/2)^\nu}{\Gamma(\frac1{2})\Gamma(\nu+\frac1{2})}\int_0^\pi e^{r \cos \theta} \sin^{2\nu}\theta \,\d\theta  \label{eq:Inu} \\
&=  \frac{(r/2)^\nu}{\Gamma(\frac1{2})\Gamma(\nu+\frac1{2})}\int_0^\pi \cosh(r \cos \theta) \sin^{2\nu}\theta \,\d\theta,  \qquad 0< r < +\infty \notag
\end{align*}
and the well known formula  
\[
\frac{\Gamma(\frac1{2})\Gamma(\nu+\frac1{2})}{\Gamma(1+\nu)} = B\left(\frac1{2},\nu+\frac1{2}\right) =  \int_0^{\pi} \sin^{N-2}\theta\,\d\theta, 
\]
Equation \eqref{eq:h}  amounts to the desired formula.

\section{Optimality of the rate of convergence for   Varadhan's equation}\label{app:rate}

We obtain an explicit rate of convergence for $-\sqrt{a} \log q_a(x)\to d(x,\partial \Omega)$ to confirm the optimality of the bound in \eqref{eq:rate_Varadhan} in  dimensions $N\ge2$.

We begin with the simplest case of one dimension and $\Omega=(-h,h)$. The solution \eqref{eq:Varadhan2} can be written as  
\[
q_a(x) = e^{\frac{x-h}{\sqrt{a}}}\cdot \frac{1+e^{-\frac{2x}{\sqrt{a}}}}{1+e^{-\frac{2h}{\sqrt{a}}}}, 
\]
so that for $0 \le x \le h$ we have 
\[
\sqrt{a} \log\frac{1}{2} \le -\sqrt{a} \log q_a(x) -(h-x) =\sqrt{a} \log\frac{1+e^{-\frac{2h}{\sqrt{a}}}}{1+e^{-\frac{2x}{\sqrt{a}}}} \le 0.   
\]
By symmetry we obtain a similar inequality for $x<0$ and amount to 
\[
\sup_{x \in \Omega}|-\sqrt{a} \log q_a(x) - d(x, \partial \Omega)| \le \sqrt{a} \log 2,  
\]
which improves \eqref{eq:rate_Varadhan}. 

In dimensions $N\ge2$, we consider the ball $\Omega := B_h(0) \subseteq \R^N$.  According to \eqref{eq:h}, the solution $q_a$ at 0 can be written as 
\begin{align*}
q_a(0) 
 = \frac{(\frac{h}{2\sqrt{a}})^{(N-2)/2}}{I_{(N-2)/2}(h/\sqrt{a})  \Gamma(\frac{N}{2})}.   
\end{align*}
Using the asymptotic bahavior $I_\nu(z) \sim \frac{e^z}{\sqrt{2\pi z}}$ as $z \to \infty~(z>0)$, see e.g.~\cite[formula 9.7.1]{AS}, we have 
\[
q_a(0) \sim  \frac{\sqrt{2\pi}}{2^{(N-2)/2}\Gamma(\frac{N}{2})} \left(\frac{h}{\sqrt{a}} \right)^{\frac{N-1}{2}}e^{-\frac{h}{\sqrt{a}}} 
\]  
and so 
\begin{equation*}\label{eq:N2}
- \sqrt{a} \log q_a(0) -h = -\frac{N-1}{4} \sqrt{a} \log \frac1{a} + \sqrt{a}\kappa_N + o(\sqrt{a}),  
\end{equation*}
where $\kappa_N:=\log\left( \Gamma (\frac{N}{2})\sqrt{\frac{2^{N-3}}{\pi h^{N-1}}} \right)$. This asymptotic formula is also valid for $N=1$ as can be checked directly.

\section*{Declarations}
\noindent
{\bf Conflict of interest.} 
On behalf of all authors, the corresponding author states that there is no conflict of interest.\\

\noindent
{\bf Data Availability.} 
This paper has no associated data.

\end{document}